\documentclass[11pt]{article}
\usepackage{amsmath}
\usepackage{amssymb}
\usepackage{amsthm}
\usepackage{cases}
\usepackage{url}

\newtheorem{theorem}{Theorem}[section]

\newtheorem{corollary}[theorem]{Corollary}
\newtheorem{proposition}[theorem]{Proposition}
\newtheorem{remark}[theorem]{Remark}

\newcommand\sa{\smallskipamount}
\newcommand\sPP{\\[\sa]\indent}
\newcommand\al\alpha
\newcommand\be\beta
\newcommand\de\delta
\newcommand\tha\theta
\newcommand\la\lambda
\newcommand\La\Lambda
\newcommand\ga{\gamma}
\newcommand\Ga{\Gamma}
\begin{document}
\title{On the differential equation for the Sobolev-Laguerre polynomials}
\author{Clemens Markett}
\date{}
\maketitle

\numberwithin{equation}{section}
\numberwithin{theorem}{section}
\begin{abstract}

The Sobolev-Laguerre polynomials form an orthogonal polynomial system with respect to a Sobolev-type inner product associated with the Laguerre measure on the positive half-axis and two point masses $M,N > 0$ at the origin involving  functions and derivatives. These polynomials have attracted much interest over the  last two decades, since they became known to satisfy, for any value of the Laguerre parameter  $\al\in\mathbb{N}_{0}$, a spectral differential equation of finite order $4\al +10$. In this paper we establish a new explicit representation of the corresponding differential operator which consists of a number of elementary components depending on $\al ,M,N$. Their interaction reveals a rich structure both being useful for applications and as a model for further investigations in the field. In particular, the Sobolev-Laguerre differential operator is shown to be symmetric with respect to the inner product.
\end{abstract}
\section{Introduction}
\label{intro}
For $\al >-1,M \ge 0,N \ge 0$, R. Koekoek and H. G. Meyer \cite{KM1}, \cite{KM2} introduced the generalized Laguerre polynomials $\{L_{n}^{\al,M,N}(x)\} _{n=0}^{\infty}\,$  , which are orthogonal on $0 \le x < \infty$ with respect to the inner product of Sobolev-type
\begin{equation}
(f,g)_{w(\al,M,N)}= \frac{1}{\Ga (\al +1)}\int_{0}^{\infty}f(x)g(x)e^{-x}x^{\al}dx+M\,f(0)g(0)+N\,f'(0)g'(0).
 \label{eq1.1}
\end{equation}  
In terms of the classical Laguerre polynomials 
\begin{equation*}
L_{n}^{\al}(x)=\frac{(\al +1)_n}{n!}\;{}_1F_1(-n;\al+1;x),\;n\in\mathbb{N}_{0},
\end{equation*}  
the Sobolev-Laguerre polynomials can be written as, cf. \cite[(10.1-2)]{KM2},
\begin{equation}
L_{n}^{\al,M,N}(x)=L_n^{\al}(x)+M\,T_n^{\al}(x)+N\,U_n^{\al}(x)+MN\,V_n^{\al}(x),\; 0 \le x < \infty,
 \label{eq1.2}
\end{equation} 	
where $T_0^{\al}(x)=U_0^{\al}(x)=U_1^{\al}(x)=V_0^{\al}(x)=V_1^{\al}(x)=0$ and, for other values of $n\in\mathbb{N}$, 
\begin{equation}
\begin{aligned}
&T_n^{\al}(x)=-\frac{(\al +2)_{n-1}}{n!}xL_{n-1}^{\al+2}(x),\\ 
&U_n^{\al}(x)=\frac{(\al +3)_{n-2}}{(\al +1)(\al +3)(n-2)!}\big\lbrack U_{n,1}^{\al}(x)+U_{n,2}^{\al}(x)+U_{n,3}^{\al}(x)\big\rbrack \; with\\ 
&U_{n,1}^{\al}(x)=\frac{x^2}{n-1}L_{n-2}^{\al+4}(x),
U_{n,2}^{\al}(x)=-(\al +2)xL_{n-1}^{\al+2}(x),
U_{n,3}^{\al}(x)=-(n+\al +1)L_n^{\al}(x),\\
&V_n^{\al}(x)=\frac{1}{\al +1}\frac{(\al +3)_{n-2}}{(n-1)!}\frac{(\al +4)_{n-2}}{n!}x^2L_{n-2}^{\al+4}(x).
 \label{eq1.3}
\end{aligned}
\end{equation}
As usual, $(a)_0=1$  and $(a)_n=a(a+1)\cdots (a+n-1), a\in \mathbb{R},\,n\in\mathbb{N}$, is the shifted factorial. Another representation of $U_n^{\al}(x)$ is given in (\ref{eq2.26}) below. For $M>0$  and $N=0$, the polynomials (\ref{eq1.2}) reduce to the Laguerre-type polynomials due to T.H. Koornwinder \cite{Ko}. 

In 1998, J. and R. Koekoek and H. Bavinck \cite{KKB} discovered that for $M,N>0$ and each $\al\in\mathbb{N}_0$, the polynomials $y_n(x)=L_{n}^{\al,M,N}(x), n\in\mathbb{N}_0$, satisfy a unique spectral differential equation of finite order  $4\al +10$. Following Bavinck \cite[(2)]{Ba1}, it is convenient to write this equation in an operational form which corresponds to the representation (\ref{eq1.2}) of the eigenfunctions,
    \begin{equation}
 \bigg\lbrace \big\lbrack \mathcal{L}_x^{\al}+\la_n^\al\big\rbrack + 
  M\big\lbrack \mathcal{A}_x^{\al}+\la_n^{\al ,A}\big\rbrack +  
  N\big\lbrack \mathcal{B}_x^{\al}+\la_n^{\al ,B}\big\rbrack +  
  MN \big\lbrack \mathcal{C}_x^{\al}+\la_n^{\al ,C}\big\rbrack\bigg\rbrace y_n(x)=0.
            \label{eq1.4}
        \end{equation}
    
The present paper deals with the four differential operators $\mathcal{L}_x^{\al},\,\mathcal{A}_x^{\al},\,\mathcal{B}_x^{\al},\,\mathcal{C}_x^{\al}$, each being independent of $n$. Our major purpose is to establish new efficient representations of the latter two associated with the parameters $M,N$. The four components of the eigenvalue parameter in (\ref{eq1.4}) are polynomials in $n$ given by, cf. \cite[Sec. 2]{KKB},
  \begin{equation}
  \begin{aligned}
  &\la_n^{\al}=n,\;
  \la_n^{\al,A}=\frac{(n)_{\al +2}}{(\al +2)!},\;
 \la_n^{\al ,B}=\frac{\al +2}{\al +1}\frac{(n-2)_{\al +4}}{(\al +4)!}+\frac{1}{\al +1}\frac{(n-1)_{\al +3}}{(\al +3)!},\\
 & \la_n^{\al ,C}=\frac{1}{\al +1}\sum_{j=0}^{\min(n-2,\al +2)}
 \frac{(\al +3-j)_{2j}}{(j!(j+1)!}\frac{(n-1-j)_{j+\al +3}}{(j+\al +3)!}.
        \label{eq1.5}
  \end{aligned}
    \end{equation}
    
For $M=N=0$, the first component of (\ref{eq1.4}) comprises the classical Laguerre equation $\big\lbrack \mathcal{L}_x^{\al}+n\big\rbrack L_n^{\al}(x)=0$, where
\begin{equation}
 \mathcal{L}_x^{\al}y(x)=\big\lbrack xD_x^2+(\al +1-x)D_x\big\rbrack y(x)=e^x x^{-\al}D_x\big\lbrack e^{-x}x^{\al +1}D_x y(x)\big\rbrack. 
 \label{eq1.6}
 \end{equation}
Here and in the following, $D_x^{i}\equiv (D_x)^{i}, i \in \mathbb{N}$, denotes an  $i$-fold differentiation with respect to $x$. The second term of (\ref{eq1.4}) is associated with the Laguerre-type equation determined by J. and R. Koekoek \cite[(15)]{KK1}. This differential equation is of order $2\al+4$ with
 \begin{equation}
 \begin{aligned}
 \mathcal{A}_x^{\al}y(x)=&\sum_{i=1}^{2\al+4}a_i(\al,x)D_x^i y(x),\\
 a_i(\al,x)=&\frac{1}{(\al +2)!}\sum_{j=\max(1,\,i-\al-2)}^{\min(i,\al+2)}
 (-1)^{i+j+1}\binom{\al+1}{j-1}\binom{\al+2}{i-j}(i+1)_{\al+2-j}\,x^j.
\label{eq1.7}
\end{aligned}
  \end{equation}
Just recently, we found another representation of this differential expression  more reminiscent of the second identity of definition (\ref{eq1.6}), see \cite{Ma2}. Indicating the order of the expression as another index, we proved that $ \mathcal{A}_x^{\al}y(x)$ coincides with 
 \begin{equation}
  \breve{\mathcal{L}}_{2\al+4,x}^{\al}y(x)=\frac{(-1)^{\al+1} }{(\al +2)!}e^x x\,D_x^{\al+2}\big\lbrace e^{-x}D_x^{\al+2}[x^{\al+1}y(x)]\big\rbrace .
  \label{eq1.8}
  \end{equation}
This result was obtained by taking a confluent limit of a similar representation of the higher-order Jacobi-type differential equation established in \cite{Ma1}. For an extension of this paper to the generalized Jacobi equation with four parameters see \cite{Ma3}. A combination of (\ref{eq1.6}) and (\ref{eq1.8}) could then be used to show that the Laguerre-type differential operator is symmetric with respect to the inner product (\ref{eq1.1}) for $N=0$.
 
In the Sobolev cases $M \ge 0,N>0$, however, the latter two components of equation (\ref{eq1.4}) are by far more complicated. In fact, it was shown by J. and R. Koekoek and H. Bavinck \cite[Cor. 2, Thm. 3]{KKB} that the coefficient functions $\{\be_{i}(\al,x)\}_{i=1}^{2\al+8}$ and $\{\ga_{i}(\al,x)\}_{i=1}^{4\al+10}$ of the corresponding differential expressions 
 \begin{equation}
  \mathcal{B}_x^{\al}y(x)=\sum_{i=1}^{2\al+8}\be_i(\al,x)D_x^i y(x),\quad  \mathcal{C}_x^{\al}y(x)=\sum_{i=1}^{4\al+10}\ga_i(\al,x)D_x^i y(x), 
 \label{eq1.9}
  \end{equation}
are polynomials of degree at most $i$, each given as a sum of up to five terms involving certain hypergeometric sums. Notably, each of the four differential operators in (\ref{eq1.4}) is of an order being twice the degree of the respective eigenvalue component in (\ref{eq1.5}). In particular, the highest coefficient functions turned out to be the monomials
 \begin{equation}
 \be_{2\al+8}(\al,x)=\frac{(\al+2)(-1)^{\al+1}}{(\al +1)(\al+4)!}x^{\al+4},\;
 \ga_{4\al+10}(\al,x)=\frac{x^{2\al+5}}{(\al+1)(2\al +5)(\al+2)!(\al+3)!}.
\label{eq1.10}
  \end{equation}
Moreover, it was proved for any $\al \in \mathbb{N}_0$, that the coefficient functions share the remarkable property 
 \begin{equation}
  \sum_{i=1}^{2\al+4}a_i(\al,x)= \sum_{i=1}^{2\al+8}\be_i(\al,x)= \sum_{i=1}^{4\al+10}\ga_i(\al,x)=0.
 \label{eq1.11}
  \end{equation}

So it is natural to raise the question whether there are useful representations of the two components $\mathcal{B}_x^\al$ and $\mathcal{C}_x^\al$  of the Sobolev-Laguerre equation as well, preferably of a form similar to the component (\ref{eq1.8}). In Section 2, we answer this question in the affirmative, see Theorem 2.1. According to their respective orders, we denote the new differential expressions by $\widetilde{\mathcal{L}}_{2\al+8,x}^{\al}y(x)$ and $\widehat{\mathcal{L}}_{4\al+10,x}^{\al}y(x)$, respectively. Surprisingly, $\widetilde{\mathcal{L}}_{2\al+8,x}^{\al}y(x)$ turns out to be a certain linear combination of three distinct parts, while the differential expression $\widehat{\mathcal{L}}_{4\al+10,x}^{\al}y(x)$  consists of $\al+3$ terms of a similar structure. As a first consequence we note that the new representations immediately satisfy the properties (\ref{eq1.10}) and (\ref{eq1.11}), as well. 
    
The proof of Theorem 2.1 crucially depends on the fact that each component of equation (\ref{eq1.4}) exhibits an eigenvalue equation by itself, whose eigenfunctions are the corresponding components of the Sobolev-Laguerre polynomial (\ref{eq1.2}). While the first of these equations is classical, see (\ref{eq1.6}), it was proved in \cite[Cor 2.3]{Ma1} that  
 \begin{equation}
  \big\lbrack \breve{\mathcal{L}}_{2\al+4,x}^\al +\la_n^{\al ,A}\big\rbrack T_n^\al (x)=0,\;\al \in \mathbb{N}_0,\;n \ge 1.  
  \label{eq1.12}
    \end{equation}
Now, in Proposition \ref{prop2.3}, we obtain, for each $\al \in \mathbb{N}_0$ and all $n \ge 2$,
 \begin{equation}
  \big\lbrack \widetilde{\mathcal{L}}_{2\al+8,x}^{\al}+\la_n^{\al ,B}\big\rbrack U_n^\al (x)=0,\;\big\lbrack \widehat{\mathcal{L}}_{4\al+10,x}^{\al}+\la_n^{\al ,C}\big\rbrack V_n^\al (x)=0.
 \label{eq1.13}
    \end{equation}
Finally, Section 3 is devoted to showing that the Sobolev-Laguerre differential operator 
\begin{equation}
  \mathcal{\mathcal{L}}_x^{\al,M,N}= \mathcal{L}_x^\al+M\breve{\mathcal{L}}_{2\al+4,x}^\al+N \widetilde{\mathcal{L}}_{2\al+8,x}^\al+MN\widehat{\mathcal{L}}_{4\al+10,x}^\al 
  \label{eq1.14}
    \end{equation}
is symmetric with respect to the inner product (\ref{eq1.1}). Consequently, the corresponding eigenfunctions form an orthogonal polynomial system in the respective function space. This enables us to trace the orthogonality of the Sobolev-Laguerre polynomials back to their differential equation property. 

Over the past quarter century, there has been an enormous interest in Sobolev orthogonal polynomials and their various properties. An excellent overview over  new developments in this fascinating area of research as well as many historical comments have been given by F. Marcell\'{a}n and Y. Xu \cite{MX}, see, in particular, the chapters on "Sobolev type orthogonal polynomials" and "Differential equations". Another profound article by W. N. Everitt, K.H. Kwon, L. L. Littlejohn, and R. Wellman \cite{EKLW} surveys the known results on orthogonal polynomial solutions of linear ordinary differential equations until the turn of the century. In particular, I. H. Jung, K. H. Kwon, D. W. Lee, and L. L. Littlejohn \cite{JKLL} found necessary and sufficient conditions for such spectral differential equations. A promising approach to Sobolev-Laguerre polynomials and their differential equations was developed by A. Gr\"{u}nbaum, L. Haine, and E. Horozov \cite{GHH} by repeatedly applying a process of Darboux transformations. Very recently, P. Iliev \cite{IL} as well as A. J. Dur\'{a}n and M. D. de la Iglesia \cite{DI} extended these results to even more general settings via new constructive techniques. 

The results of the present paper and the method of proof  certainly give some deeper insight into the nature of higher-order differential equations with polynomial eigenfunctions and should be worthwhile for further studies. For instance, it is very likely that there are similar results in case of the differential equation for Sobolev-type Gegenbauer polynomials, cf. \cite{Ba2}, \cite{BK} and the literature cited there. 
 
  \section{New representation of the Sobolev-Laguerre equation}
     \label{sec:2}
 
 The main result of this paper is the following. 
 \begin{theorem}
 \label{thm2.1}
 For $\al \in \mathbb{N}_0,M \ge 0, N>0$, let $\mathcal{L}_x^{\al,M,N}$ denote the Sobolev-Laguerre differential operator (\ref{eq1.14}) with the (combined) eigenvalue defined via (\ref{eq1.5}) by
   \begin{equation}
   \la_n^{\al,M,N}=\la_n^\al+M\la_n^{\al ,A} +N\la_n^{\al ,B}+MN\la_n^{\al ,C},\;n \in \mathbb{N}_0.
         \label{eq2.1}
     \end{equation}
 Then the Sobolev-Laguerre polynomials $y_n(x)=L_n^{\al,M,N}(x),\,n \in \mathbb{N}_0,$  satisfy the spectral differential equation
   \begin{equation}
   \mathcal{L}_x^{\al,M,N}y_n(x)=-\la_n^{\al,M,N}y_n(x),\;n \in \mathbb{N}_0,
         \label{eq2.2}
     \end{equation}
where the differential expressions  $\mathcal{L}_x^{\al}y(x)$ and $\breve{\mathcal{L}}_{2\al+4,x}^{\al}y(x)$ are given in (\ref{eq1.6}) and (\ref{eq1.8}), respectively. Furthermore, 
    \begin{equation}
    \widetilde{\mathcal{L}}_{2\al+8,x}^{\al}y(x)=\frac{(-1)^{\al}}{(\al +1)(\al +4)!}\big\lbrack \mathcal{E}_x^\al +\mathcal{F}_x^\al +\mathcal{G}_x^\al \big\rbrack y(x)
          \label{eq2.3}
      \end{equation}
 with
 \begin{equation}
   \begin{aligned}
  \mathcal{E}_x^\al y(x)&=-(\al+2)\,e^x x\,D_x^{\al+4}\big\lbrace e^{-x} x^2 D_x^{\al+4}\big\lbrack x^{\al+1}y(x)\big\rbrack\big\rbrace,\\
   \mathcal{F}_x^\al y(x)&=(\al+4)\,e^x x\,D_x^{\al+3}\big\lbrace e^{-x} (x+2\al+4)\,D_x^{\al+3}\big\lbrack x^{\al+1}y(x)\big\rbrack\big\rbrace,\\
    \mathcal{G}_x^\al y(x)&=(\al+1)(\al+3)(\al+4)\,e^x D_x^{\al+2}\big\lbrace e^{-x} (x+2\al+4)\,D_x^{\al+2}\big\lbrack x^{\al} y(x)\big\rbrack\big\rbrace. 
   \end{aligned}
  \label{eq2.4}
  \end{equation}
 and
  \begin{equation}
    \begin{aligned}
   \widehat{\mathcal{L}}_{4\al+10,x}^{\al} y(x)&=\sum_{j=0}^{\al +2}\frac{(\al+3-j)_{2j}}{j!(j+1)!}\mathcal{H}_x^{\al,j}y(x),\\
   \mathcal{H}_x^{\al,j}y(x)&=\frac{(-1)^{\al+j}}{(\al +1)(\al+3+j)!}e^x x\,D_x^{\al+3+j}\big\lbrace e^{-x}x^j (x+q_j^\al)\,D_x^{\al+3+j}\big\lbrack x^{\al+1}y(x)\big\rbrack\big\rbrace,
  \end{aligned}
   \label{eq2.5}
   \end{equation}
 where
   \begin{equation}
    q_j^\al =(\al +2)^{-1}(\al +2-j)(\al +3+j).
    \label{eq2.6}
    \end{equation}
  \end{theorem}

 Before we give the proof of this theorem which is based on Proposition \ref{prop2.3} below, we state some of its consequences.
 
For the first three values of the parameter, $\al =0,1,2$, formula (\ref{eq2.5})--(\ref{eq2.6}) states that 
 \begin{equation}
   \begin{aligned}
  e^{-x} x^{-1}\widehat{\mathcal{L}}_{10,x}^{\,0}&y(x)=\frac{2}{5!}D_x^5\big\lbrace e^{-x} x^3 D_x^5\big\lbrack xy(x)\big\rbrack\big\rbrace-
  \frac{3}{4!}D_x^4\big\lbrace e^{-x} x(x+2)D_x^4\big\lbrack xy(x)\big\rbrack\big\rbrace\\
  & +\frac{1}{3!}D_x^3\big\lbrace e^{-x}(x+3)D_x^3\big\lbrack xy(x)\big\rbrack\big\rbrace,\\
  2e^{-x} x^{-1}\widehat{\mathcal{L}}_{14,x}^{\,1}&y(x)=\frac{5}{7!}D_x^7
  \big\lbrace e^{-x} x^4 D_x^7\big\lbrack x^2 y(x)\big\rbrack\big\rbrace-
    \frac{10}{6!}D_x^6\big\lbrace e^{-x} x^2(x+2)D_x^6\big\lbrack x^2 y(x) \big\rbrack\big\rbrace\\
   & +\frac{6}{5!}D_x^5\big\lbrace e^{-x} x(x+\frac{10}{3})D_x^5\big\lbrack x^2y(x)\big\rbrack\big\rbrace-
     \frac{1}{4!}D_x^4\big\lbrace e^{-x}(x+4)D_x^4\big\lbrack x^2 y(x) \big\rbrack\big\rbrace,\\
  3e^{-x} x^{-1}\widehat{\mathcal{L}}_{18,x}^{\,2}&y(x)=\frac{14}{9!}D_x^9
  \big\lbrace e^{-x} x^5 D_x^9\big\lbrack x^3 y(x)\big\rbrack\big\rbrace-
     \frac{35}{8!}D_x^8\big\lbrace e^{-x} x^3(x+2)D_x^8\big\lbrack x^3 y(x) \big\rbrack\big\rbrace\\
   & +\frac{30}{7!}D_x^7\big\lbrace e^{-x} x^2(x+\frac{7}{2})D_x^7\big\lbrack x^3y(x) \big\rbrack\big\rbrace-
      \frac{10}{6!}D_x^6\big\lbrace e^{-x}x(x+\frac{9}{2})D_x^6\big\lbrack x^3 y(x) \big\rbrack\big\rbrace\\ 
  & +\frac{1}{5!}D_x^5\big\lbrace e^{-x} x(x+5)D_x^5\big\lbrack x^3y(x) \big\rbrack\big\rbrace. 
    \end{aligned}
  \label{eq2.7}
  \end{equation}
It is not hard to see that in these three cases, the new representations (2.7) coincide with those stated in \cite[Sec. 4.1--3]{K1}, see also \cite{K2}. 

Expanding the differential expressions (\ref{eq2.3})--(\ref{eq2.4}) and (\ref{eq2.5})--(\ref{eq2.6}) into the series 
 \begin{equation}
  \widetilde{\mathcal{L}}_{2\al+8,x}^{\al}y(x)=\sum_{i=1}^{2\al+8} b_i^\al(x)D_x^i y(x),\quad \widehat{\mathcal{L}}_{4\al+10,x}^{\al} y(x)=\sum_{i=1}^{4\al+10} c_i^\al(x)D_x^i y(x), 
 \label{eq2.8}
  \end{equation}
the highest coefficient functions obviously arise in the differential expressions   $\mathcal{E}_x^\al y(x)$ and
 \begin{equation*}
\mathcal{H}_x^{\al,\al+2}y(x)=\frac{1}{(\al +1)(2\al+5)!}e^x x\,D_x^{2\al+5}\big\lbrace e^{-x}x^{\al+3} D_x^{2\al+5}\big\lbrack x^{\al+1}y(x)\big\rbrack\big\rbrace,
 \end{equation*}
respectively. Hence, we obtain
 \begin{equation}
 b_{2\al+8}^\al(x)=\frac{(\al+2)(-1)^{\al+1}}{(\al +1)(\al+4)!}x^{\al+4},\quad
 c_{4\al+10}^\al(x)=\frac{(1)_{2\al+4}}{(\al+2)!(\al+3)!}\frac{x^{2\al+5}}{(\al+1) (2\al +5)!}.
\label{eq2.9}
  \end{equation}
These values coincide with $\be_{2\al+8}(\al,x),\,\ga_{4\al+10}(\al,x)$ stated in (\ref{eq1.10}). Similarly, it follows that 
 \begin{equation}
 \breve{\mathcal{L}}_{2\al+4,x}^\al y(x)=\sum_{i=1}^{2\al+4} a_i^\al(x)D_x^i y(x),\;   a_{2\al+4}^\al(x)\equiv a_{2\al+4}(\al,x)=\frac{(-1)^{\al+1}}{(\al+2)!}x^{\al+2}.
 \label{eq2.10}
  \end{equation}
Finally, we conclude that the series (\ref{eq2.8}) and (\ref{eq2.10}) satisfy the property (\ref{eq1.11}), as well.
\begin{corollary}
\label{cor2.2}
 For each $\al \in \mathbb{N}_0$, 
 \begin{equation}
   \begin{aligned}
 &\sum_{i=1}^{2\al+4} a_i^\al(x)=e^{-x}\breve{\mathcal{L}}_{2\al+4,x}^\al [e^x]=0, \\
 & \sum_{i=1}^{2\al+8} b_i^\al(x)=e^{-x}\widetilde{\mathcal{L}}_{2\al+8,x}^{\al}[e^x]=0,\\
 &\sum_{i=1}^{4\al+10} c_i^\al(x)=e^{-x}\widehat{\mathcal{L}}_{4\al+10,x}^{\al}[e^x]=0. 
   \end{aligned}
 \label{eq2.11}
  \end{equation}
\end{corollary}

\begin{Proof}
For $0 \le r<t \le s,$ there holds
 \begin{equation*}
D_x^s\big\lbrace e^{-x}x^r D_x^s\big\lbrack x^{s-t}e^x\big\rbrack\big\rbrace
=D_x^s\bigg\lbrace \sum_{k=t}^{s}\binom{s}{k}\frac{(s-t)!}{(k-t)!}x^{k-t+r}\bigg\rbrace =0.
 \end{equation*}
Hence, all three identities in (2.11) are verified by observing that
\begin{equation*}
   \begin{aligned}
 & D_x^{\al+2}\big\lbrace e^{-x}D_x^{\al+2}\big\lbrack x^{\al+1} e^x\big\rbrack\big\rbrace=0, \\
 & D_x^{\al+4}\big\lbrace e^{-x}x^2 D_x^{\al+4}\big\lbrack x^{\al+1} e^x\big\rbrack\big\rbrace=0,\,D_x^{\al+2+j}\big\lbrace e^{-x}(x+2\al+4)\, D_x^{\al+2+j}\big\lbrack x^{\al+j} e^x\big\rbrack\big\rbrace=0,\;j=0,1,\\
 & D_x^{\al+3+j}\big\lbrace e^{-x}x^j(x+q_j^\al)\, D_x^{\al+3+j}\big\lbrack x^{\al+1} e^x\big\rbrack\big\rbrace=0,\;j=0,\cdots \al+2.\hspace{4,5cm}\square
   \end{aligned}
  \end{equation*}
 \end{Proof}
\begin{proposition}
\label{prop2.3}
For $\al \in \mathbb{N}_0$ and all $n \ge 2$, the components $U_n^\al(x)$  and   $V_n^\al(x)$ of the Sobolev-Laguerre polynomials (1.2) satisfy the spectral differential equations (1.13), i. e.
 \begin{equation}
   \begin{aligned}
 & \widetilde{\mathcal{L}}_{2\al +8,x}^\al U_n^\al(x)=-\la_n^{\al,B}\,U_n^\al(x),\\
 & \widehat{\mathcal{L}}_{4\al +10,x}^\al V_n^\al(x)=-\la_n^{\al,C}\,V_n^\al(x).
     \end{aligned}
 \label{eq2.12}
  \end{equation}
\end{proposition}
\begin{Proof}
Concerning the first equation in (\ref{eq2.12}), we split up the corresponding component of the eigenvalue parameter by 
\begin{equation*}
\la_n^{\al,B}=\frac{(-1)^\al}{(\al+1)(\al+4)!}\big\lbrack \la_{n,1}^{\al,B}+\la_{n,2}^{\al,B} \big\rbrack,\;\bigg\lbrace
\begin{aligned}
&\la_{n,1}^{\al,B}=(-1)^\al(\al+2)(n-2)_{\al+4} \\
&\la_{n,2}^{\al,B}=(-1)^\al(\al+4)(n-1)_{\al+3}
\end{aligned}
\end{equation*}
as well as the expressions $\mathcal{F}_x^\al y(x)$ and $\mathcal{G}_x^\al y(x)$  in definition (\ref{eq2.4}) by 
 \begin{equation*}
 \begin{aligned}
     &\mathcal{F}_{x,1}^\al y(x)=(\al+4)e^x x\,D_x^{\al+3}\big\lbrace e^{-x} x\,D_x^{\al+3}\big\lbrack x^{\al+1}y(x)\big\rbrack\big\rbrace,\\
   &\mathcal{F}_{x,2}^\al y(x)=2(\al+2)(\al+4)e^x x\,D_x^{\al+3}\big\lbrace e^{-x} D_x^{\al+3}\big\lbrack x^{\al+1}y(x)\big\rbrack\big\rbrace,\\
  &\mathcal{G}_{x,1}^\al y(x)=(\al+1)(\al+3)(\al+4)\,e^x D_x^{\al+2}\big\lbrace e^{-x} x\,D_x^{\al+2}\big\lbrack x^{\al} y(x)\big\rbrack\big\rbrace,\\
  &\mathcal{G}_{x,2}^\al y(x)=2(\al+1)_4\,e^x D_x^{\al+2}\big\lbrace e^{-x} D_x^{\al+2}\big\lbrack x^{\al} y(x)\big\rbrack\big\rbrace. 
  \end{aligned}
       \end{equation*}
Then, by definition (\ref{eq1.3}) of $U_n^\al(x)$, it is to show that 
\begin{equation}
 \big\lbrack \mathcal{E}_x^\al +\mathcal{F}_{x,1}^\al+\mathcal{F}_{x,2}^\al+\mathcal{G}_{x,1}^\al+\mathcal{G}_{x,2}^\al
 +\la_{n,1}^{\al,B}+\la_{n,2}^{\al,B}\,\big\rbrack
 \big\lbrack U_{n,1}^\al(x)+U_{n,2}^\al(x)+U_{n,3}^\al(x)\big\rbrack=0.
  \label{eq2.13}
  \end{equation}
Applying each operation in the first bracket separately to the three functions in the second one, we have to evaluate no fewer than 21 different pieces, namely, for $i=1,2$ and $j=1,2,3$, 
\begin{equation*}
 E_j =\mathcal{E}_x^\al\,(U_{n,j}^\al,x),\,F_{i,j} =\mathcal{F}_{x,i}^\al\,(U_{n,j}^\al,x),\,
 G_{i,j} =\mathcal{G}_{x,i}^\al\,(U_{n,j}^\al,x),\,\La_{i,j} =\la_{n,i}^{\al,B}\cdot  U_{n,j}^\al(x).
  \end{equation*}
To this end, we frequently make use of some well-known recurrence and differentiation formulas for the Laguerre polynomials, cf. e. g. \cite[10.12]{HTF2}. For suited values of $\ga,n$, these are
\begin{equation}
 \begin{aligned}
 & L_n^\ga(x)=L_n^{\ga+1}(x)-L_{n-1}^{\ga+1}(x),\quad
 x\,L_n^\ga(x)=(n+\ga)L_n^{\ga-1}(x)-(n+1)L_{n+1}^{\ga-1}(x),\\
 & D_x L_n^\ga(x)=-L_{n-1}^{\ga+1}(x),\quad
  D_x \lbrack x^\ga L_n^\ga(x)\rbrack =(n+\ga)x^{\ga-1}L_n^{\ga-1}(x),\\
  &  D_x \lbrack e^{-x} L_n^\ga(x)\rbrack=-e^{-x} L_n^{\ga+1}(x), \quad
   D_x \lbrack e^{-x}x^\ga L_n^\ga(x)\rbrack =(n+1)
   e^{-x}x^{\ga-1}L_{n+1}^{\ga-1}(x),\\ 
 \end{aligned}
   \label{eq2.14}
  \end{equation}
and, consequently, $L_n^{\ga+2}(x)=\sum_{k=0}^{n}L_k^{\ga+1}(x)=\sum_{k=0}^{n}(n+1-k)L_k^\ga(x)$. First of all, we observe that problem (\ref{eq2.13}) slightly simplifies because of
\begin{equation}
 E_2 +\La _{1,2}=0,\quad F_{1,2}+\La _{2,2}=0,\quad F_{1,3}+G_{1,3}+\La _{2,3}=0.
  \label{eq2.15}
  \end{equation}
In fact, by combining the two formulas
 \begin{equation*}
 \begin{aligned}
 &D_x^{\al+4} \lbrack x^{\al +2} L_{n-1}^{\al+2}(x)\rbrack =(n)_{\al +2}D_x^2 L_{n-1}^0(x)=(n)_{\al +2}L_{n-3}^2(x),\\
&D_x^{\al+4} \lbrack e^{-x}x^2 L_{n-3}^2(x)\rbrack =(-1)^\al (n-2)_2\,e^{-x}L_{n-1}^{\al +2}(x),
  \end{aligned}
       \end{equation*}
we obtain 
 \begin{equation*}
 \begin{aligned}
 E_2&=-(\al+2)\,e^x x\,D_x^{\al+4} \big\lbrace e^{-x}x^2 D_x^{\al+4} \lbrack -(\al +2)x^{\al+2} L_{n-1}^{\al+2}(x) \rbrack \big\rbrace\\
 &=(-1)^\al (\al+2)^2 (n-2)_{\al+4}\,x\,L_{n-1}^{\al +2}(x)=-\La_{1,2}.
  \end{aligned}
       \end{equation*}
The second identity in (\ref{eq2.15}) holds analogously, and, omitting some intermediate steps, the third one follows by
\begin{equation*}
\begin{aligned}
 \mathcal{F}_{x,1}&L_n^\al(x)+ \mathcal{G}_{x,1}L_n^\al(x)+ \la_{n,2}^{\al ,B}\,L_n^\al(x)\\
 =&(-1)^\al (\al +4)(n+1)_\al \bigg\lbrace  x\,\big\lbrack (n-1)nL_{n-1}^{\al +2}(x)
 +(\al +1)nL_{n-1}^{\al +3}(x)-(\al +1)L_{n-1}^{\al+4}(x)\big\rbrack\\
 &-(\al +1)(\al +3)\,\big\lbrack nL_{n-1}^{\al +2}(x)-L_{n-1}^{\al +3}(x)\big\rbrack
 +(n-1)n(n+\al +1)L_n^\al (x) \bigg\rbrace =0.
 \end{aligned}
 \end{equation*}
Concerning all other pieces, our strategy is first to expand them into Laguerre series with parameter $\al+2$ and then to cluster the resulting terms appropriately. For instance, we have
\begin{equation}
\begin{aligned}
&G_{1,1}+G_{1,2}=\frac{1}{n-1} \mathcal{G}_{x,1}^\al\big\lbrack x^2 L_{n-2}^{\al+4}(x) \big\rbrack -(\al +2)\,\mathcal{G}_{x,1}^\al \big\lbrack  x L_{n-1}^{\al +2}(x)\big\rbrack \\
&=(-1)^{\al +1}(\al +1)(\al +3)(\al +4) \;\bigg\lbrace\sum_{s=0}^{n-2} \frac{n-1-s}{n-1}(s+1)_{\al +2}\;\cdot\\
&\hspace{1cm} \big\lbrack (s+1)L_{s+1}^{\al +1}(x)-sL_s^{\al +1}(x)\big\rbrack
+(\al +2) \sum_{s=0}^{n-2}(s+1)_{\al +2}L_{s+1}^{\al +1}(x)\bigg\rbrace\\
&=(-1)^\al(\al +1)(\al +3)^2(\al +4)\;\bigg\lbrace \frac{(\al +2)!}{n-1}L_0^{\al +2}(x)\\
&\hspace{1cm} +\sum_{s=1}^{n-2}\frac{(\al+3)s+\al+2}{n-1}(s+1)_{\al+1} L_s^{\al+2}(x) -(n-1)_{\al +2}L_{n-1}^{\al +2}(x) \bigg\rbrace.
 \end{aligned}
  \label{eq2.16}
 \end{equation}
Moreover we obtain by some tedious, but straightforward calculations, each guided by a numerical verification for small values of the parameters,  
\begin{equation}
\begin{aligned}
&E_1+F_{1,1}+\La_{1,1}+\La_{2,1}
=(-1)^\al(\al +3)(\al +4)\;\bigg\lbrace \frac{(\al+3)(\al+3)!}{n-1}L_0^{\al +2}(x)\\
&+\sum_{s=1}^{n-2}\frac{\phi(\al ,s)}{(\al +4)(n-1)}(s+1)_{\al +1}L_s^{\al +2}(x)-
  \big\lbrack (\al +2)n-\al -1 \big\rbrack (n-1)_{\al +2}L_{n-1}^{\al +2}(x) \bigg\rbrace \,,
 \end{aligned}
  \label{eq2.17}
 \end{equation}
where
\begin{equation*}
\begin{aligned}
\phi(\al ,s)&=\big\lbrack(\al +2)s-\al \big\rbrack (s-1)s^2-\big\lbrack(\al+2)s+2 \big\rbrack (2s+\al+3)s(s+\al+2)\\
&+\big\lbrack(\al +2)s+\al+4 \big\rbrack (s+\al+2)(s+\al+3)^2 \\
&=(\al+3)(\al+4)\big\lbrack 2(\al +2)s^2+(\al+3)^2s+(\al+2)(\al+3)\big\rbrack,
 \end{aligned}
  \end{equation*}
and
\begin{equation}
\begin{aligned}
&(E_3+\La_{1,3})+F_{2,1}+(F_{2,2}+F_{2,3}+G_{2,1}+G_{2,2}+G_{2,3})
=(-1)^{\al+1}(\al +2)_3\,\cdot\\
&\bigg\lbrace \frac{2(\al+3)!}{n-1}L_0^{\al +2}(x)+\sum_{s=1}^{n-2}
\frac{2(\al+3)(s+1)}{n-1}(s+1)_{\al +2}L_s^{\al +2}(x)-
  (n-1)_{\al +3}L_{n-1}^{\al +2}(x) \bigg\rbrace \,.
 \end{aligned}
  \label{eq2.18}
 \end{equation}
Collecting the coefficients of $L_0^{\al +2}(x)$ as well as of $L_{n-1}^{\al +2}(x)$ in (\ref{eq2.16})--(\ref{eq2.18}), both sums clearly vanish. Finally we notice that for each $1 \le s \le n-2$ , the coefficients of $L_s^{\al +2}(x)$ add up to 
\begin{equation*}
\begin{aligned}
(-1)^\al\frac{(\al+3)(\al+4)}{n-1}
&\bigg\lbrace (\al +1)(\al +3)\big\lbrack(\al +3)s+\al+2 \big\rbrack \\
&-2(\al+2)(\al+3)(s+1)(s+\al +2)+\frac{1}{\al +4}\phi(\al ,s)\bigg\rbrace
(s+1)_{\al+1} \,.
 \end{aligned}
   \end{equation*}
Since the factor in curly brackets is always zero, we arrive at the first equation in (\ref{eq2.12}).

As to the second one, we have to show that
\begin{equation*}
\widehat{\mathcal{L}}_{4\al +10,x}^\al \big\lbrack x^2L_{n-2}^{\al+4}(x)\big\rbrack\\
=-\la_n^{\al,C}x^2L_{n-2}^{\al+4}(x),\;n \ge 2. 
\end{equation*}
By definition (\ref{eq2.5}) and (\ref{eq2.6}), 
\begin{equation}
\begin{aligned}
\frac{\al +1}{x}&\widehat{\mathcal{L}}_{4\al +10,x}^\al \big\lbrack x^2L_{n-2}^{\al+4}(x) \big\rbrack\\
&=\sum_{j=0}^{\al +2}\frac{(\al+3-j)_{2j}(-1)^{\al +j}\,e^x}{j!(j+1)!(\al +3+j)!}
   \,D_x^{\al+3+j}\big\lbrace e^{-x}x^{j+1}\,D_x^{\al+3+j}\big\lbrack x^{\al+3}L_{n-2}^{\al +4}(x)\big\rbrack\big\rbrace\\
&+\sum_{j=0}^{\al +2}\frac{(\al+2-j)_{2j+1}(-1)^{\al +j}\,e^x}{j!(j+1)!(\al +2+j)!(\al +2)}
   \,D_x^{\al+3+j}\big\lbrace e^{-x}x^j\,D_x^{\al+3+j}\big\lbrack x^{\al+3}L_{n-2}^{\al +4}(x)\big\rbrack\big\rbrace.   
 \end{aligned}
  \label{eq2.19}
\end{equation}
Employing again our tool box (\ref{eq2.14}), we find that
\begin{equation*}
\begin{aligned}
D_x^{\al+3+j}\big\lbrack x^{\al+3}L_{n-2}^{\al +4}(x)\big\rbrack&=
\sum_{r=0}^{n-2}D_x^j\,D_x^{\al+3}\big\lbrack x^{\al+3}L_r^{\al +3}(x)\big\rbrack\\
&=\sum_{r=0}^{n-2}\frac{(r+\al +3)!}{r!}D_x^j L_r^0(x)=(-1)^j
\sum_{r=j}^{n-2}\frac{(r+\al +3)!}{r!}L_{r-j}^j(x)
 \end{aligned}
\end{equation*}
and thus
\begin{equation}
\begin{aligned}
(-1)^{\al +j}&e^x\,D_x^{\al+3+j}\big\lbrace e^{-x}x^j\,D_x^{\al+3+j}\big\lbrack x^{\al+3}L_{n-2}^{\al +4}(x)\big\rbrack\big\rbrace\\
&=(-1)^\al \sum_{r=j}^{n-2}\frac{(r+\al+3)!}{r!}e^x\,D_x^{\al+3}\,D_x^j\big\lbrack e^{-x}x^jL_{r-j}^j(x)\big\rbrack\\   
&=(-1)^\al \sum_{r=j}^{n-2}\frac{(r+\al+3)!}{(r-j)!}e^x\,D_x^{\al+3}\big\lbrack e^{-x}L_r^0(x)\big\rbrack= -\sum_{r=j}^{n-2}\frac{(r+\al+3)!}{(r-j)!} L_r^{\al+3}(x).  
 \end{aligned}
  \label{eq2.20}
\end{equation}
Similarly,
\begin{equation}
\begin{aligned}
(-1)^{\al +j}&e^x\,D_x^{\al+3+j}\big\lbrace e^{-x}x^{j+1}\,D_x^{\al+3+j}\big\lbrack x^{\al+3}L_{n-2}^{\al +4}(x)\big\rbrack\big\rbrace\\
&=(-1)^\al \sum_{r=j}^{n-2}\frac{(r+\al+3)!}{r!}e^x\,D_x^{\al+2}\,D_x^{j+1}\big\lbrack e^{-x}x^{j+1}\{L_{r-j}^{j+1}(x)-L_{r-j-1}^{j+1}(x)\}\big\rbrack\\   
&=\sum_{r=j}^{n-2}\frac{(r+\al+3)!}{(r-j)!} \big\lbrack (r+1)L_{r+1}^{\al +2}(x)-(r-j)L_r^{\al +2}(x)\big\rbrack\\ 
&=\sum_{r=j}^{n-2}\frac{(r+\al+3)!}{(r-j)!} \big\lbrack -x\,L_r^{\al +3}(x)+(\al +3+j)L_r^{\al +2}(x)\big\rbrack.
 \end{aligned}
  \label{eq2.21}
\end{equation}
Inserting (\ref{eq2.21}) and (\ref{eq2.20}) into the identity (\ref{eq2.19}) then yields 
\begin{equation*}
\begin{aligned}
\frac{\al +1}{x}&\widehat{\mathcal{L}}_{4\al +10,x}^\al \big\lbrack x^2L_{n-2}^{\al+4}(x) \big\rbrack\\
&=-x\sum_{j=0}^{\min(n-2,\al +2)} \frac{(\al +3-j)_{2j}}{j!(j+1)!(\al +3+j)!}
\sum_{r=j}^{n-2}\frac{(r+\al +3)!}{(r-j)!}L_r^{\al +3}(x)\\
&+\sum_{j=0}^{\min(n-2,\al +2)} \frac{(\al +3-j)_{2j}}{j!(j+1)!(\al +2+j)!}
\sum_{r=j}^{n-2}\frac{(r+\al +3)!}{(r-j)!}L_r^{\al +2}(x)\\
&-\sum_{j=0}^{\min(n-2,\al +2)} \frac{(\al +2-j)_{2j+1}}{j!(j+1)!(\al +2+j)!(\al +2)} \sum_{r=j}^{n-2}\frac{(r+\al +3)!}{(r-j)!}L_r^{\al +3}(x)\\
&=:\Sigma_1+\Sigma_2+\Sigma_3.
 \end{aligned}
 \end{equation*}
Now we split up the inner sum of $\Sigma_1$ into the two parts
\begin{equation*}
\begin{aligned}
&\sum_{r=j}^{n-2}\frac{(r+\al +3)!}{(r-j)!}\big\lbrack L_r^{\al +4}(x)-L_{r-1}^{\al +4}(x)\big\rbrack\\
&=\frac{(n+\al +1)!}{(n-2-j)!}L_{n-2}^{\al +4}(x)+\sum_{r=\max(j,1)}^{n-2}
\bigg\lbrack \frac{(r+\al +2)!}{(r-1-j)!}-\frac{(r+\al +3)!}{(r-j)!}\bigg\rbrack 
L_{r-1}^{\al +4}(x)\\
&=\frac{(n+\al +1)!}{(n-2-j)!}L_{n-2}^{\al +4}(x)-(\al +3+j)\sum_{r=\max(j,1)}^{n-2}
\frac{(r+\al +2)!}{(r-j)!}L_{r-1}^{\al +4}(x).
 \end{aligned}
 \end{equation*}
Hence,
\begin{equation*}
\Sigma_1=-x\sum_{j=0}^{\al +2} \frac{(\al +3-j)_{2j}\,(n-1-j)_{j+\al+3}}{j!(j+1)!(\al +3+j)!}L_{n-2}^{\al +4}(x)+\Sigma_{1,2},
 \end{equation*}
where the first term equals $-(\al+1)\la_n^{\al,C}x\,L_{n-2}^{\al+4}(x)$  as required. The second term becomes
\begin{equation*}
\begin{aligned}
\Sigma_{1,2}&=x\sum_{j=0}^{\min(n-2,\al +2)} \frac{(\al +3-j)_{2j}}{j!(j+1)!(\al +2+j)!}\sum_{r=\max(j,1)}^{n-2}
\frac{(r+\al +2)!}{(r-j)!}L_{r-1}^{\al +4}(x)\\
&=\sum_{r=1}^{n-2}\bigg\lbrack \sum_{j=0}^{\min(r,\al+2)}\frac{(-\al -2)_j(-r)_j}{(2)_j\,j!}\bigg\rbrack\frac{(r+\al+2)!}{r!(\al+2)!}xL_{r-1}^{\al +4}(x).
 \end{aligned}
 \end{equation*}
Applying now the well-known Chu-Vandermonde summation formula 
\begin{equation}
{}_2F_1(-m,b;c;x)=(c-b)_m/(c)_m,\;c>0,\,m \in \mathbb{N}_0
  \label{eq2.22}
 \end{equation}
to the inner sum, we achieve
\begin{equation*}
\begin{aligned}
\Sigma_{1,2}&=\sum_{r=1}^{n-2}\frac{(r+\al +3)!}{(r+1)!(\al+3)!}
\frac{(r+\al +2)!}{r!(\al+2)!}xL_{r-1}^{\al +4}(x)\\
&=\sum_{r=1}^{n-2}\frac{(r+\al +3)!}{(r+1)!(\al+3)!}
\frac{(r+\al +2)!}{r!(\al+2)!}\big\lbrack (r+\al +3)L_{r-1}^{\al +3}(x)-rL_r^{\al +3}(x)\big\rbrack.
 \end{aligned}
 \end{equation*}
Analoguously,
\begin{equation*}
\begin{aligned}
\Sigma_2&=\sum_{r=0}^{n-2}\frac{(r+\al +3)!}{(r+1)!(\al+3)!}
\frac{(r+\al +3)!}{r!(\al+2)!}L_r^{\al +2}(x)\\
&=\sum_{r=0}^{n-2}\frac{(r+\al +3)!}{(r+1)!(\al+3)!}
\frac{(r+\al +3)!}{r!(\al+2)!}\big\lbrack L_r^{\al +3}(x)-L_{r-1}^{\al+3}(x) \big\rbrack
 \end{aligned}
 \end{equation*}
and 
\begin{equation*}
\begin{aligned}
\Sigma_3&=-\sum_{r=0}^{n-2}\bigg\lbrack \sum_{j=0}^{\min(r,\al+1)}\frac{(-\al -1)_j(-r)_j}{(2)_j)j!}\bigg\rbrack \frac{(r+\al +3)!}{r!(\al+2)!}L_r^{\al +3}(x)\\
&=-\sum_{r=0}^{n-2}\frac{(r+\al +2)!}{(r+1)!(\al+2)!}
\frac{(r+\al +3)!}{r!(\al+2)!} L_r^{\al +3}(x).
 \end{aligned}
 \end{equation*}
So, putting all parts together, we arrive at
\begin{equation*}
\begin{aligned}
\Sigma_{1,2}+\Sigma_2+\Sigma_3=&
\sum_{r=0}^{n-2}\frac{(r+\al +2)!}{(r+1)!(\al+2)!}
\frac{(r+\al +3)!}{r!(\al+3)!}\cdot\\
&\big\lbrack -r+(r+\al+3)-(\al +3)\big\rbrack L_r^{\al+3}(x)=0.
\end{aligned}
\end{equation*}
This concludes the proof of Proposition \ref{prop2.3}.$\hspace{7cm}\square$
\end{Proof}\\

\noindent\textbf{Proof of Theorem 2.1.}\; Inserting the Sobolev-Laguerre polynomials (\ref{eq1.2}) into the known Sobolev-Laguerre equation (\ref{eq1.4}) and comparing the terms with equal powers of the parameters $M$ and $N$, one obtains a system of identities including 
 \begin{equation}
\big\lbrack \mathcal{B}_x^\al +\la_n^{\al,B}\big\rbrack U_n^\al(x)=0,\;
\big\lbrack \mathcal{C}_x^\al +\la_n^{\al,C}\big\rbrack V_n^\al(x)=0,\;n \ge 2.
 \label{eq2.23}
  \end{equation}
So in view of Proposition \ref{prop2.3}, it follows that
 \begin{equation}
\big\lbrack \mathcal{B}_x^\al -\widetilde{\mathcal{L}}_{2\al +8,x}^\al \big\rbrack U_n^\al(x)=0,\;
\big\lbrack \mathcal{C}_x^\al -\widehat{\mathcal{L}}_{4\al +10,x}^\al\big\rbrack V_n^\al(x)=0,\;n \ge 2.
 \label{eq2.24}
  \end{equation}
Moreover, by (\ref{eq1.9}, (\ref{eq1.11}) and the definition (\ref{eq2.3})--(\ref{eq2.5}), the two identities (\ref{eq2.24}) hold as well, when the functions $U_n^\al(x)$ and $V_n^\al(x)$ are replaced by any linear function. This, however, implies that the differential operators $\mathcal{B}_x^\al$  and $\widetilde{\mathcal{L}}_{2\al +8,x}^\al$ as well as $\mathcal{C}_x^\al$ and $\widehat{\mathcal{L}}_{4\al +10,x}^\al$  coincide on the set of all algebraic polynomials $\mathbb{P}$. In fact, it just remains to show that for any $p_n(x) \in \mathbb{P}$, there are constants $c_k=c_k(\al,n)$ and $d_k=d_k(\al,n),\,0 \le k \le n$, such that 
 \begin{equation}
p_n(x)=c_0+c_1x+\sum_{k=2}^{n}c_kU_k^\al(x)=d_0+d_1x+\sum_{k=2}^{n}d_kV_k^\al(x).
 \label{eq2.25}
  \end{equation}
Concerning the first identity, we start off from the Laguerre expansion 
  \begin{equation*}
 p_n(x)=\sum_{k=0}^{n}h_k^\al(p_n,L_k^\al)_{w(\al)}L_k^\al(x),\;h_k^\al =(L_k^\al,L_k^\al)_{w(\al)}^{-1}\,.
    \end{equation*}
Next we observe that the polynomials $U_k^\al(x),\,k \ge 2$, satisfy, cf. \cite[(1),(2)]{KKB},
\begin{equation}
\begin{aligned}
&\frac{(\al +1)(\al +3)(k-2)!}{(\al +3)_{k-2}}U_k^\al (x)\\
&=\big\lbrack k(\al +2)-(\al +1)\big\rbrack L_k^\al (x)+(\al +2)(\al +3)
\bigg\lbrack D_xL_k^\al(x)+\frac{1}{k-1}D_x^2L_k^\al(x) \bigg\rbrack\\
&=\big\lbrack k(\al+2)-(\al+1)\big\rbrack L_k^\al (x)-\frac{(\al +2)(\al +3)}{k-1}\sum_{j=1}^{k-1}j\,L_j^\al(x),
 \end{aligned}
  \label{eq2.26}
\end{equation}
because
\begin{equation*}
\begin{aligned}
&D_xL_k^\al(x)+\frac{1}{k-1}D_x^2L_k^\al(x) =-L_{k-1}^{\al+1}(x)+\frac{1}{k-1}L_{k-2}^{\al+2}(x)\\
&=-\sum_{j=0}^{k-1}L_j^\al(x)+\frac{1}{k-1}\sum_{j=0}^{k-2}(k-1-j)L_j^\al(x)
=-\frac{1}{k-1}\sum_{j=1}^{k-1}j\,L_j^\al(x).
 \end{aligned}
 \end{equation*}
By inverting formula (\ref{eq2.26}), each Laguerre polynomial $L_k^\al(x),\,k \ge 2$, can be expressed as a linear combination of a linear function and the polynomials $U_j^\al(x),\,2 \le j \le k$. This implies the first identity in (\ref{eq2.25}). The second one is verified even simpler. Indeed, setting $p_n(x)=p_n(0)+p'_n(0)x+x^2q_{n-2}(x)$ with some $q_{n-2}(x) \in \mathbb{P}_{n-2}$, we note that 
\begin{equation*}
x^2q_{n-2}(x)=x^2\sum_{k=2}^{n}h_{k-2}^{\al+4}(q_{n-2},L_{k-2}^{\al+4})_{w(\al+4)} L_{k-2}^{\al+4}(x)=:\sum_{k=2}^{n}d_k(\al ,n)V_k^\al(x).
 \end{equation*}
So, together with the first two, all components of the Sobolev-Laguerre differential equation in Theorem \ref{thm2.1} coincide with those in the  representation (\ref{eq1.4}).\hfill $\square$ 

 \section{Symmetry of the Sobolev-Laguerre differential operator and orthogonality of the corresponding eigenfunctions}
 \label{sec:3}

One great advantage to be gained from the new representation of the Sobolev-Laguerre differential operator $\mathcal{L}_x^{\al,M,N}$  is that it gives rise to the following symmetry property.
 \begin{theorem}
 \label{thm3.1}
 Let $\al \in \mathbb{N}_0$  and $M \ge 0,\,N>0$.\\ 
 \textbf{(i)} The operator $\mathcal{L}_x^{\al,M,N}$ is symmetric with respect to the inner product (\ref{eq1.1}), i.e.
\begin{equation}
\big(\mathcal{L}_x^{\al,M,N}f,g\big)_{w(\al,M,N)} =\big(f,\mathcal{L}_x^{\al,M,N}g\big)_{w(\al,M,N)},\;f,g \in C^{(4\al+10)}[0,\infty).
  \label{eq3.1}
\end{equation}
\textbf{(ii)} The polynomial eigenfunctions of the Sobolev-Laguerre equation (\ref{eq2.2}), $y_n(x)=L_n^{\al,M,N}(x)$, $n \in \mathbb{N}_0$, satisfy the orthogonality relation 
\begin{equation}
\big(y_n,y_m\big)_{w(\al,M,N)}=0,\;n \ne m,\,n,m \in \mathbb{N}_0.
  \label{eq3.2}
\end{equation}
 \end{theorem}

The proof requires some preliminary calculations. 
 \begin{proposition}
 \label{prop3.2}
  \textbf{(i)} For $\al \in \mathbb{N}_0$  and $f,g \in 
  C^{(4\al +10)}[0,\infty)$ we define the following integrals each being symmetric in the two functions $f,g$,
 \begin{equation*}
 \begin{aligned}
 I_1^\al(f,g)=&\frac{1}{\al!} \int_{0}^{\infty}e^{-x}x^{\al +1}f'(x)g'(x)dx,\\
 I_2^\al(f,g)=&\frac{1}{\al!(\al +2)!} \int_{0}^{\infty}e^{-x}D_x^{\al +2} \big\lbrack x^{\al +1}f(x)\big\rbrack D_x^{\al +2}\big\lbrack x^{\al +1}g(x)\big\rbrack dx,\\
   I_{3,1}^\al(f,g)=&\frac{\al +2}{(\al+1)!(\al +4)!} \int_{0}^{\infty}e^{-x}
  x^2D_x^{\al +4} \big\lbrack x^{\al +1}f(x)\big\rbrack D_x^{\al +4}\big\lbrack x^{\al +1}g(x)\big\rbrack dx,\\ 
  I_{3,2}^\al(f,g)=&\frac{1}{(\al+1)!(\al +3)!} \int_{0}^{\infty}e^{-x}
  (x+2\al +4) D_x^{\al +3} \big\lbrack x^{\al +1}f(x)\big\rbrack D_x^{\al +3}\big\lbrack x^{\al +1}g(x)\big\rbrack dx,\\ 
   I_{3,3}^\al(f,g)=&\frac{1}{\al!(\al +2)!} \int_{0}^{\infty}e^{-x}
   (x+2\al +4) D_x^{\al +2} \big\lbrack x^\al f(x)\big\rbrack D_x^{\al +2}\big\lbrack x^\al g(x)\big\rbrack dx, 
 \end{aligned}
  \end{equation*}
 and, for $0 \le j \le \al+2$ and $q_j^\al=(\al+2)^{-1}(\al +2-j)(\al +3+j)$ as defined in (\ref{eq2.6}),
  \begin{equation*}
     I_{4,j}^\al(f,g)=\frac{1}{(\al+1)!(\al +3+j)!} \int_{0}^{\infty}e^{-x}
    x^j(x+q_j^\al) D_x^{j+\al +3} \big\lbrack x^{\al+1} f(x)\big\rbrack D_x^{j+\al +3}\big\lbrack x^{\al +1} g(x)\big\rbrack dx. 
   \end{equation*}	 
 Moreover we set
\begin{equation*}
 \begin{aligned}
 &(S_1f)(0)=\sum_{j=0}^{\al +2}\frac{(-1)^j(\al +3-j)_{2j}(\al +2+j)}{j!(j+2)!}f^{(j+2)}(0),\\
 &(S_2f)(0)=\sum_{j=1}^{\al +2}\frac{(-1)^{j+1}(\al +3-j)_{2j}}{j!(j+1)!}f^{(j+1)}(0).
  \end{aligned}
 \end{equation*}	 
Recalling the definition (\ref{eq1.14}) of $\mathcal{L}_x^{\al,M,N}$, the corresponding differential expressions satisfy 
 \begin{equation}
\big(\mathcal{L}_x^\al f,g\big)_{w(\al)}=
-I_1^\al (f,g),\hspace{8,2cm}
  \label{eq3.3}
\end{equation}
  \begin{equation}
 \big(\breve{\mathcal{L}}_{2\al+4,x}^\al f,g\big)_{w(\al)}=
 -I_2^\al (f,g)-(\al+1)f'(0)g(0),\hspace{5cm}
   \label{eq3.4}
 \end{equation}
  \begin{equation}
 \big(\widetilde{\mathcal{L}}_{2\al+8,x}^\al f,g\big)_{w(\al)}=
 -\sum_{j=1}^{3}
 I_{3,j}^\al (f,g)-(\al+2)f''(0)g'(0),\hspace{4cm}
   \label{eq3.5}
 \end{equation}
   \begin{equation}
  \big(\widehat{\mathcal{L}}_{4\al+10,x}^\al  f,g\big)_{w(\al)}=
  -\sum_{j=0}^{\al +2} \frac{(\al +3-j)_{2j}}{j!(j+1)!}I_{4,j}^\al (f,g)-(S_1f)(0)g(0)-(S_2f)(0)g'(0).
     \label{eq3.6}
  \end{equation}
   \textbf{(ii)} At the origin, the differential expressions and their derivatives take the values
   \begin{equation}
  \big(\mathcal{L}_x^\al f\big)(0)=(\al +1)f'(0),\hspace{1cm}
  \big(\mathcal{L}_x^\al f\big)'(0)=(\al +2)f''(0)-f'(0),\hspace{1cm}
    \label{eq3.7}
  \end{equation}
    \begin{equation}
   \big(\breve{\mathcal{L}}_{2\al+4,x}^\al f\big)(0)=0,\hspace{2,1cm}
    \big(\breve{\mathcal{L}}_{2\al+4,x}^\al f\big)'(0)=-f'(0)+ \big(S_2 f\big)(0),\hspace{0,2cm}
     \label{eq3.8}
   \end{equation}
    \begin{equation}
   \big(\widetilde{\mathcal{L}}_{2\al+8,x}^\al f\big)(0)=\big(S_1 f\big)(0),\hspace{0,8cm}
   \big(\widetilde{\mathcal{L}}_{2\al+8,x}^\al f\big)'(0)=0,\hspace{3,2cm}
     \label{eq3.9}
   \end{equation}
     \begin{equation}
    \big(\widehat{\mathcal{L}}_{4\al+10,x}^\al  f\big)(0)=0,\hspace{2cm}
    \big(\widehat{\mathcal{L}}_{4\al+10,x}^\al  f\big)'(0)=0.\hspace{3cm}
       \label{eq3.10}
    \end{equation}
    \end{proposition}
  \begin{Proof} We repeatedly utilize that, for any sufficiently differentiable function $f(x)$,
   \begin{equation*}
   D_x^s[x^tf(x)]\, \vert_{x=0}=\frac{s!}{(s-t)!}f^{(s-t)}(0),\;s,t \in \mathbb{N}_0,\,s \ge t \ge 0.
        \end{equation*}
 If, in addition, $q,r \in \mathbb{N}_0,\,q \ge r \ge 0$, we have
  \begin{equation}
   \begin{aligned}
   \frac{1}{q!}&D_x^q\Big\lbrace e^{-x}x^rD_x^s[x^tf(x)]\Big\rbrace \, \big\vert_{x=0}=
   \sum_{k=0}^{q}\frac{(-1)^{q-k}}{k!(q-k)!}D_x^k\Big\lbrace x^rD_x^s[x^tf(x)]\Big\rbrace \, \big\vert_{x=0}\\
   &=\sum_{k=0}^{q}\frac{(-1)^{q-k}}{(q-k)!}\sum_{j=0}^{k}\frac{1}{j!(k-j)!}
   D_x^{k-j} \big[x^r\big]\,D_x^{s+j}\big[x^tf(x)\big] \, \big\vert_{x=0}\\
    &=\sum_{k=r}^{q}\frac{(-1)^{q-k}}{(q-k)!}\frac{1}{(k-r)!}\frac{(k+s-r)!}
    {(k+s-r-t)!}f^{(k+s-r-t)}(0)\\
   &=\sum_{k=0}^{q-r}\frac{(-1)^{q-r-k}}{(q-r-k)!\,k!}\frac{(k+s)!}
      {(k+s-t)!}f^{(k+s-t)}(0).  
            \label{eq3.11}
    \end{aligned} 
      \end{equation}
   \textbf{(i)} The identities (\ref{eq3.3}) and (\ref{eq3.4}) have been proved in \cite[Sec.4]{Ma1} via integration by parts, i.e.,
    \begin{equation*}
     \begin{aligned}
    \big(\mathcal{L}_x^\al f,g\big)_{w(\al)}&=\frac{1}{\al!} \int_{0}^{\infty}D_x
   \big\lbrack e^{-x}x^{\al +1}D_xf(x)\big\rbrack\, g(x)dx
   =-\frac{1}{\al!} \int_{0}^{\infty}e^{-x}x^{\al +1}f'(x)g'(x)dx,\\
        \big(\breve{\mathcal{L}}_{2\al+4,x}^\al f,g\big)_{w(\al)}&=
     \frac{(-1)^{\al +1}}{\al!(\al +2)!} \int_{0}^{\infty}D_x^{\al +2}
      \big\lbrace e^{-x}D_x^{\al +2} \big\lbrack x^{\al +1}f(x)\big\rbrack \big\rbrace \,x^{\al +1} g(x)dx  \\
   =-(\al +1)&f'(0)g(0)-\frac{1}{\al!(\al +2)!} \int_{0}^{\infty}
         e^{-x}D_x^{\al +2} \big\lbrack x^{\al +1}f(x)\big\rbrack \,D_x^{\al +2} \big\lbrack x^{\al +1}g(x)\big\rbrack dx. 
         \end{aligned} 
        \end{equation*}
  Concerning (\ref{eq3.5}), we find, in view of (\ref{eq2.3}), (\ref{eq2.4}), that
   \begin{equation*}
  \big(\widetilde{\mathcal{L}}_{2\al+8,x}^\al f,g\big)_{w(\al)}=\Omega_1+\Omega_2+\Omega_3
       \end{equation*}
   with
      \begin{equation*}
       \begin{aligned}
      \Omega_1&=\frac{(\al+2)(-1)^{\al+1}}{(\al+1)!(\al+4)!} \int_{0}^{\infty}
      D_x^{\al +4} \big\lbrace e^{-x}x^2D_x^{\al +4} \big\lbrack x^{\al +1}f(x)\big\rbrack \big\rbrace \,x^{\al +1} g(x)dx  \\
     &=\sum_{k=0}^{\al +3}\frac{(\al +2)(-1)^{k+\al +1}}{(\al+1)!(\al+4)!} 
    D_x^{\al +3-k} \big\lbrace e^{-x}x^2 D_x^{\al +4} \big\lbrack x^{\al +1}f(x)\big\rbrack \big\rbrace\, D_x^k \big\lbrack x^{\al +1}g(x)\big\rbrack \big\vert _{x=0}^{\infty}-I_{3,1}^\al (f,g).
           \end{aligned} 
          \end{equation*}
 In the sum, the only non-vanishing part occurs for $k=\al +1$, evaluated at $x=0$, i. e., 
  \begin{equation}
  -\frac{\al +2}{(\al+1)!(\al+4)!}  D_x^2 \big\lbrace e^{-x}x^2 D_x^{\al +4} \big\lbrack x^{\al +1}f(x)\big\rbrack \big\rbrace\, D_x^{\al +1} \big\lbrack x^{\al +1}g(x) \big\rbrack \big\vert _{x=0}=-\frac{\al +2}{3}f^{(3)}(0)g(0).
   \label{eq3.12}
   \end{equation}
   Furthermore,
     \begin{equation*}
    \begin{aligned}
   \Omega_2&=\frac{(-1)^\al}{(\al+1)!(\al+3)!} \int_{0}^{\infty}
   D_x^{\al +3} \big\lbrace e^{-x}(x+2\al +4) D_x^{\al +3} \big\lbrack x^{\al +1}f(x)\big\rbrack \big\rbrace \,x^{\al +1} g(x)dx  \\
   &=\sum_{k=0}^{\al+2}\frac{(-1)^{k+\al}}{(\al+1)!(\al+3)!} 
   D_x^{\al +2-k} \big\lbrace e^{-x}(x+2\al +4) D_x^{\al +3} \big\lbrack x^{\al +1}f(x)\big\rbrack \big\rbrace\, D_x^k \big\lbrack x^{\al +1}g(x)\big\rbrack \big\vert _{x=0}^{\infty}\\
   &\quad -I_{3,2}^\al (f,g).
             \end{aligned} 
            \end{equation*}
  Here, only the terms for $k=\al+1$ and $k=\al+2$, evaluated at $x=0$, contribute to the sum by 
   \begin{equation}
    \begin{aligned}
  &\frac{1}{(\al+1)!(\al+3)!} 
  D_x \big\lbrace e^{-x}(x+2\al+4) D_x^{\al+3} \big\lbrack x^{\al+1}f(x)\big\rbrack \big\rbrace\, D_x^{\al+1} \big\lbrack x^{\al+1}g(x)\big\rbrack \big\vert _{x=0}\\
  &=\frac{1}{(\al+3)!} 
    \big\lbrace (2\al+4) D_x^{\al+4} \big\lbrack x^{\al+1}f(x)\big\rbrack-
    (2\al+3) D_x^{\al+3} \big\lbrack x^{\al+1}f(x)\big\rbrack
     \big\rbrace \big\vert _{x=0}\,g(0)\\
  &=\big\lbrack \frac{1}{3}(\al +2)(\al +4) f^{(3)}(0)-\frac{1}{2}(2\al +3) f''(0)
       \big\rbrack g(0)  
      \end{aligned}
      \label{eq3.13}
     \end{equation}
   and, respectively,
    \begin{equation}
   -\frac{2\al +4}{(\al+1)!(\al+3)!} D_x^{\al+3} \big\lbrack x^{\al+1}f(x)\big\rbrack \, D_x^{\al+2} \big\lbrack x^{\al+1}g(x)\big\rbrack \big\vert _{x=0} =-(\al+2)^2 f''(0)g'(0).     
         \label{eq3.14}
        \end{equation}
  Analogously, there holds
      \begin{equation*}
       \begin{aligned}
      \Omega_3&=\frac{(-1)^\al}{\al!(\al+2)!} \int_{0}^{\infty}
      D_x^{\al +2} \big\lbrace e^{-x}(x+2\al +4) D_x^{\al +2} \big\lbrack x^\al f(x)\big\rbrack \big\rbrace \,x^\al g(x)dx  \\
      &=\sum_{k=0}^{\al +1}\frac{(-1)^{k+\al}}{\al!(\al+2)!}
      D_x^{\al +1-k} \big\lbrace e^{-x}(x+2\al +4) D_x^{\al +2} \big\lbrack x^{\al}f(x)\big\rbrack \big\rbrace\, D_x^k \big\lbrack x^{\al}g(x)\big\rbrack \big\vert _{x=0}^{\infty}-I_{3,3}^\al (f,g),
          \end{aligned} 
        \end{equation*} 
   where the sum reduces to
      \begin{equation}
     -\big\lbrack \frac{1}{3}(\al +2)(\al +3) f^{(3)}(0)-\frac{1}{2}(2\al +3) f''(0)
     \big\rbrack g(0)+(\al +1)(\al +2)f''(0)g'(0).     
           \label{eq3.15}
          \end{equation}
   Hence, the terms (\ref{eq3.12})--(\ref{eq3.15}) add up to $-(\al +2)f''(0)g'(0)$ as asserted.
   
   Similarly, we obtain for any $j \in \mathbb{N}_0$, 
      \begin{equation*}
     \begin{aligned}
     &\big(\mathcal{H}_x^{\al,j}f,g\big)_{w(\al)}\\
     &=\frac{(-1)^{j+\al}}{(\al+1)!(j+\al+3)!}
     \int_{0}^{\infty}D_x^{j+\al+3} \big\lbrace e^{-x}x^j(x+q_j^\al) D_x^{j+\al+3} \big\lbrack x^{\al +1}f(x)\big\rbrack \big\rbrace x^{\al+1} g(x)dx\\
     &=\sum_{k=0}^{j+\al+2}\!\frac{(-1)^{k+j+\al}}{(\al+1)!(j+\al+3)!}
     D_x^{j+\al +2-k} \big\lbrace e^{-x}x^j(x+q_j^\al) D_x^{j+\al+3} \big\lbrack x^{\al +1} f(x)\big\rbrack \big\rbrace\, D_x^k \big\lbrack x^{\al+1} g(x)\big\rbrack \big\vert _{x=0}^{\infty}\\
     &\quad -I_{4,j}^\al (f,g).
     \end{aligned} 
     \end{equation*} 
    Again, all but the summands for $k=\al+1$ and $k=\al+2$, evaluated at $x=0$,  vanish. For $k=\al+1$, we obtain           
        \begin{equation*}
          \begin{aligned}
       &\frac{(-1)^j}{(j+\al+3)!}
       D_x^{j+1} \big\lbrace e^{-x}x^j(x+q_j^\al) D_x^{j+\al+3} \big\lbrack x^{\al +1} f(x)\big\rbrack \big\rbrace\, \big\vert _{x=0}\,g(0)\\
       &=\frac{(-1)^j}{j+2}(1-q_j^\al)f^{(j+2)}(0)g(0)+\frac{(-1)^j(\al +4+j)}{(j+2)(j+3)}q_j^\al f^{(j+3)}(0)g(0)\\
       &=\frac{(-1)^{j+1}(\al+3-j)(\al+2+j)}{(j+2)(\al +2)} f^{(j+2)}(0)g(0)\\
       &\quad +\frac{(-1)^j(\al+4+j)(\al+2-j)(\al+3+j)}{(j+2)(j+3)(\al +2)} f^{(j+3)}(0)g(0)
        \end{aligned} 
        \end{equation*} 
       and, for $k=\al +2$,
      \begin{equation*}
      \begin{aligned}
      &\frac{(-1)^{j+1}}{(\al +1)!(j+\al+3)!}
     D_x^j \big\lbrace e^{-x}x^j(x+q_j^\al) D_x^{j+\al+3} \big\lbrack x^{\al +1} f(x)\big\rbrack \big\rbrace\,D_x^{\al+2} \big\lbrack x^{\al +1} g(x)\big\rbrack\,\big\vert _{x=0}\\
     &=\frac{(-1)^{j+1}(\al+2-j)(\al+3+j)}{(j+1)(j+2)} f^{(j+2)}(0)g'(0).
                  \end{aligned} 
              \end{equation*} 
     Putting all parts together then yields
      \begin{equation*}
     \begin{aligned}
     &\big(\widehat{\mathcal{L}}_{4\al+10,x}^\al f,g\big)_{w(\al)}= \sum_{j=0}^{\al+2}\frac{(\al+3-j)_{2j}}{j!(j+1)!}\big(\mathcal{H}_x^{\al,j} f,g \big)_{w(\al)} \\
     &=-\sum_{j=0}^{\al+2}\frac{(\al+3-j)_{2j}}{j!(j+1)!}I_{4,j}^\al (f,g)
     +g(0)\bigg\lbrace \sum_{j=0}^{\al+2}\frac{(-1)^j(\al+3-j)_{2j}}{j!(j+2)!}
     f^{(j+2)}(0)-\\
     &\quad-\sum_{j=0}^{\al+2}\frac{(-1)^j(\al+2-j)_{2j+2}}{j!(j+2)!(\al+2)}  f^{(j+2)}(0)+\sum_{j=0}^{\al+2}\frac{(-1)^j(\al+2-j)_{2j+2}(\al +4+j)} {j!(j+3)!(\al+2)}f^{(j+3)}(0) \bigg\rbrace \\
     &\quad-\sum_{j=0}^{\al+1}\frac{(-1)^j(\al+2-j)_{2j+2}}{(j+1)!(j+2)!}
     f^{(j+2)}(0)g'(0).
     \end{aligned} 
     \end{equation*} 
    Notice that the three sums in curly brackets can be combined by means of an index transformation in the third one and the fact that
      \begin{equation*}
     \begin{aligned}
    &\frac{(-1)^j(\al+3-j)_{2j}\big[\al+2-(\al+3+j)(\al+2-j)-(\al+3+j)\,j\big]}
   {j!(j+2)!(\al+2)}\\
   &=\frac{(-1)^{j+1}(\al+3-j)_{2j}(\al+2+j)}{j!(j+2)!}.
            \end{aligned} 
         \end{equation*} 
   By another index transformation in the last sum we arrive at the identity (\ref{eq3.6}).\\
   
    \textbf{(ii)} The values (\ref{eq3.7}) follow at once by definition (\ref{eq1.6}) of $\mathcal{L}_x^\al$. In (\ref{eq3.8}), $\big(\breve{\mathcal{L}}_{2\al+4,x}^\al f\big)(0)=0$ is trivial, while in view of (\ref{eq3.11}),
         \begin{equation*}
       \begin{aligned}
   \big(\breve{\mathcal{L}}_{2\al+4,x}^\al f\big)'(0)=& \frac{(-1)^{\al+1}}{(\al+2)!} D_x^{\al+2}\big\lbrace e^{-x}D_x^{\al+2}\big\lbrack x^{\al +1} f(x)\big\rbrack \big\rbrace\, \big\vert _{x=0}\\
   =&\sum_{k=0}^{\al+2}\frac{(-1)^{k+1}(\al+3-k)_{2k}}{k!(k+1)!} f^{(k+1)}(0)=-f'(0)+\big(S_2 f\big)(0).
    \end{aligned} 
           \end{equation*} 
    Concerning (\ref{eq3.9}), 
     \begin{equation*}
      \begin{aligned}
      \big(\widetilde{\mathcal{L}}_{2\al+8,x}^\al f\big)(0)&=\frac{(-1)^\al\big(\mathcal{G}_x^\al f\big)(0)}{(\al+1)(\al+4)!} = \frac{(-1)^\al}{(\al+2)!}
      D_x^{\al+2}\big\lbrace e^{-x}(x+2\al +4)D_x^{\al+2}\big\lbrack x^\al f(x)\big\rbrack \big\rbrace\, \big\vert _{x=0}\\
      &=\sum_{k=0}^{\al+2}\frac{(-1)^k(\al+3-k)_{2k}(2 \al +4-[\al +2-k])}{k!(k+2)!} f^{(k+2)}(0)=\big(S_1 f\big)(0).
       \end{aligned} 
              \end{equation*} 
      We further notice that
         \begin{equation*}
          \begin{aligned}
      &\big(\mathcal{E}_x^\al f\big)'(0)=\frac{\al +2}{\al +1}          \sum_{k=1}^{\al+3}\frac{(-1)^k(k+ \al +3)!}{(k-1)!(k+2)!(\al +3-k)!} f^{(k+2)}(0),\\
      & \big(\mathcal{F}_x^\al f\big)'(0)=\frac{1}{\al +1}          \sum_{k=0}^{\al+3}\frac{(-1)^{k+1}(k+ \al +1)(k+ \al+3)!}{k!(k+2)!(\al +3-k)!} f^{(k+2)}(0),\\
      &\big(\mathcal{G}_x^\al f\big)'(0)= \sum_{k=0}^{\al+2}\frac{(-1)^k}{k!(\al +2-k)!} D_x^{k+1}\big\lbrace (x+2\al +4)D_x^{\al+2}\big\lbrack x^\al f(x)\big\rbrack \big\rbrace\, \big\vert _{x=0}\\ 
      &=\sum_{k=0}^{\al+2}\frac{(-1)^k}{k!(\al+2-k)!}\bigg\lbrace (2\al+4)
      \frac{(k+\al+3)!}{(k+3)!}f^{(k+3)}(0)+(k+1) \frac{(k+\al+2)!}
      {(k+2)!} f^{(k+2)}(0)\bigg\rbrace \\
     &=\sum_{k=0}^{\al+3}\frac{(-1)^{k+1}(k+\al+2)!}{k!(k+2)!(\al+3-k)!}\big\lbrack (2\al+4)k-(k+1)(\al +3-k)\big\rbrack f^{(k+2)}(0)\\ 
     &=\sum_{k=0}^{\al+3}\frac{(-1)^{k+1}(k-1)(k+\al+3)!}{k!(k+2)!(\al+3-k)!}
     f^{(k+2)}(0).
     \end{aligned} 
           \end{equation*} 
      Since $k(\al+2)/(\al+1)-(k+\al+1)/(\al+1)-(k-1)=0$, $k \in \mathbb{N}_0$, we conclude that
       \begin{equation*}
       \big(\widetilde{\mathcal{L}}_{2\al+8,x}^\al f\big)'(0)=\frac{(-1)^\al}{(\al+1)(\al+4)!}\big\lbrace \big(\mathcal{E}_x^\al f\big)'(0)+\big(\mathcal{F}_x^\al f\big)'(0)+\big(\mathcal{G}_x^\al f\big)'(0)\big\rbrace =0.
      \end{equation*} 
      
      Concerning (\ref{eq3.10}), it is clear that $ \big(\widehat{\mathcal{L}}_{4\al+10,x}^\al f\big)(0)=0$. Moreover, 
      \begin{equation*}
      \big(\widehat{\mathcal{L}}_{4\al+10,x}^\al f\big)'(0)= \sum_{j=0}^{\al+2}\frac{(\al+3-j)_{2j}}{j!(j+1)!}\big(\mathcal{H}_x^{\al,j}  f \big)'(0) 
     \end{equation*} 
      with
      \begin{equation*}
      \begin{aligned}
     &\big(\mathcal{H}_x^{\al,j} f \big)'(0)= 
       \frac{(-1)^{\al+j}}{(\al +1)(\al+3+j)!}
      D_x^{\al+3+j}\big\lbrace e^{-x}x^j(x+q_j^\al)D_x^{\al+3+j}\big\lbrack x^{\al+1} f(x)\big\rbrack \big\rbrace\, \big\vert _{x=0}\\
     &=\sum_{k=j+1}^{\al+3+j}\frac{(-1)^{k+1}}{(\al+1)(\al+3+j-k)!}\bigg\lbrace 
      \frac{(k+\al+2)!}{(k-j-1)!(k+1)!}f^{(k+1)}(0)+\frac{q_j^\al (k+\al+3)!}
      {(k-j)!(k+2)!} f^{(k+2)}(0)\bigg\rbrace \\
     &=\sum_{k=j}^{\al+3+j}\frac{(-1)^k}{(\al+1)(\al+3+j-k)!}
          \frac{(k+\al+3)!\,(\al+3+j-k-q_j^\al)}{(k-j)!(k+2)!}f^{(k+2)}(0).
       \end{aligned} 
      \end{equation*} 
      Hence, by interchanging the order of summation,
     \begin{equation*}
      \big(\widehat{\mathcal{L}}_{4\al+10,x}^\al f\big)'(0)=  \sum_{k=0}^{2\al+4}b_k^\al\frac{(-1)^k(k+\al+3)!}{(\al+1)(k+2)!}f^{(k+2)}(0), 
          \end{equation*} 
      where
    \begin{equation*}
   \begin{aligned}
    b_k^\al=&\sum_{j=\max(0,k-\al-3)}^{\min(\al+2,k)} 
   \frac{(\al +3-j)_{2j}}{j!(j+1)!(k-j)!}
   \frac{\big\lbrack \al+3-k+j-(\al+2-j)(j+\al+3)/(\al+2)\big\rbrack}{(\al+3-k+j)!}\\
   =&\sum_{j=\max(0,k-\al-2)}^{\min(\al+2,k)} 
      \frac{(\al +3-j)_{2j}}{j!(j+1)!(k-j)!(\al+2-k+j)!}\\
    &\quad -\sum_{j=\max(0,k-\al-3)}^{\min(\al+1,k)} 
         \frac{(\al +2-j)_{2j+2}}{j!(j+1)!(k-j)!(\al+3-k+j)!(\al +2)}\,.  
         \end{aligned}
    \end{equation*} 
    To see that $\big(\widehat{\mathcal{L}}_{4\al+10,x}^\al f\big)'(0)=0$, we must show that the coefficients $b_k^\al$ vanish for all $0 \le k \le 2\al+4$ or, equivalently,
    \begin{equation}
   \sum_{j=\max(0,k-\al-2)}^{\min(\al+2,k)} 
   \frac{(j+\al +2)!(-k)_j(-\al -2)_j}{j!(j+1)!(j-k+\al +2)!}=
    \sum_{n=\max(0,k-\al-3)}^{\min(\al+1,k)} 
      \frac{(n+\al+3)!(-k)_n(-\al -1)_n}{n!(n+1)!(n-k+\al+3)!} .
    \label{eq3.16}
     \end{equation}
    By employing the Chu-Vandermonde formula (\ref{eq2.22}), we observe that for all $\max(0,k-\al-2) \le j \le \min(\al+2,k)$,
   \begin{equation}
      C_{j,k}^\al :=\sum_{n=\max(0,k-j-1)}^{\min(\al+1,k)} 
     \frac{(-k)_n(-\al -1)_n}{(n-k+j+1)!\,n!} =\frac{(j+\al +2)!}{(j+1)!(j-k+\al +2)!}\,.
        \label{eq3.17}
         \end{equation}
    Inserting (\ref{eq3.17}) into the left-hand side of (\ref{eq3.16}), interchanging the order of summation, and using (\ref{eq3.17}) again with $\al$  replaced by $\al+1$, we arrive at the required result
     \begin{equation*}
      \begin{aligned} 
   &\sum_{j=\max(0,k-\al-2)}^{\min(\al+2,k)}  C_{j,k}^\al
     \frac{(-k)_j(-\al -2)_j}{j!}\\
     &\quad =\sum_{j=\max(0,k-\al-2)}^{\min(\al+2,k)} \bigg \lbrack
      \sum_{n=\max(0,k-j-1)}^{\min(\al+1,k)}\frac{(-k)_n(-\al -1)_n}{(1-k+j+n)!\,n!}
      \bigg \rbrack \frac{(-k)_j(-\al -2)_j}{j!}\\
   \hspace {2cm} &\quad =\sum_{n=\max(0,k-\al-3)}^{\min(\al+1,k)} \bigg \lbrack
    \sum_{j=\max(0,k-n-1)}^{\min(\al+2,k)}\frac{(-k)_j(-\al -2)_j}{(1-k+j+n)!\,j!}
    \bigg \rbrack \frac{(-k)_n(-\al -1)_n}{n!}\\
     &\quad = \sum_{n=\max(0,k-\al-3)}^{\min(\al+1,k)}  C_{n,k}^{\al+1}
          \frac{(-k)_n(-\al -1)_n}{n!}\,. \hspace {6cm}\square 
          \end{aligned} 
               \end{equation*}
    \end{Proof}
   \begin{remark}
   For $0 \le k \le \al+2$, identity (\ref{eq3.16}) may be written in the form
      \begin{equation}
    {}_3F_2 \left(\begin{matrix} \al+3,-\al-2,-k\\
    2,\al+3-k \end{matrix};1\right
    )=\frac{\al +3}{\al +3-k} {}_3F_2 \left(\begin{matrix}\al+4,-\al-1,-k\\
    2,\al+4-k \end{matrix};1\right)
               \label{eq3.18}
            \end{equation}
  This is a terminating version of Thomae's  ${}_3F_2-$transformation formula \cite[(2)]{Wh} %
         \begin{equation*}
           \begin{aligned} 
       {}_3F_2 \left(\begin{matrix}a,b,c\\d,e \end{matrix};1\right)
       =\frac{\Ga(d)\Ga(e)\Ga(s)}{\Ga(a)\Ga(s+b)\Ga(s+c)}
       {}_3F_2 \left(\begin{matrix}d-a,e-a,s\\s+b,s+c \end{matrix};1\right)\,,\;s=d+e-a-b-c.
     \end{aligned} 
     \end{equation*}
  \end{remark}

\noindent\textbf{Proof of Theorem 3.1. (i)}\; In view of Proposition \ref{prop3.2} (ii), the scalar product on the left-hand side of (\ref{eq3.1}) is given by
      \begin{equation*}
             \begin{aligned} 
 &\big(\mathcal{L}_x^{\al,M,N}f,g\big)_{w(\al ,M,N)}=\big(\mathcal{L}_x^{\al,M,N}f,g\big)_{w(\al)}
 +M \big(\mathcal{L}_x^{\al,M,N}f)(0)g(0)+N \big(\mathcal{L}_x^{\al,M,N}f)'(0)g'(0)\\
 &=\big(\mathcal{L}_x^\al f,g\big)_{w(\al)}+M \big\lbrack \big(\breve{\mathcal{L}}_{2\al+4,x}^\al f,g\big)_{w(\al)}+\big(\mathcal{L}_x^\al f)(0)g(0) \big\rbrack
 +N \big\lbrack \big(\widetilde{\mathcal{L}}_{2\al+8,x}^\al f,g\big)_{w(\al)}\\
 &\quad +\big(\mathcal{L}_x^\al f)'(0)g'(0) \big\rbrack +MN \big\lbrack\big(\widehat
 {\mathcal{L}}_{4\al+10,x}^{\al}f,g\big)_{w(\al)}+\big(\widetilde{\mathcal{L}}_{2\al+8,x}^\al f \big)(0)g(0)+\big(\breve{\mathcal{L}}_{2\al+4,x}^\al f\big)'(0)g'(0) \big\rbrack\\
 &=-I_1^\al(f,g)-M \big\lbrack I_2^\al(f,g)+(\al +1)f'(0)g(0)-(\al +1)f'(0)g(0)  \big\rbrack\\
 &\quad -N \Big\lbrack \sum_{j=1}^{3}I_{3,j}^\al(f,g)+(\al +2)f''(0)g'(0)-\big\lbrace(\al +2)f''(0)-f'(0)\big\rbrace g'(0)\Big\rbrack\\
 &\quad -MN \Big\lbrack \sum_{j=0}^{\al +2}\frac{(\al+3-j)_{2j}}{j!(j+1)!}
 I_{4,j}^\al(f,g)+(S_1f)(0)g(0)+(S_2f)(0)g'(0)-\\
 &\hspace{2cm}-(S_1f)(0)g(0)+\big\lbrace f'(0)-(S_2f)(0)\big\rbrace g'(0)\Big\rbrack .
   \end{aligned} 
   \end{equation*}
 Hence we achieve the identity
   \begin{equation}
  \begin{aligned} 
  \big(\mathcal{L}_x^{\al,M,N}f,g\big)_{w(\al ,M,N)}=&-I_1^\al(f,g)-M I_2^\al(f,g) -N\sum_{j=1}^{3}I_{3,j}^\al(f,g)\\
  &-MN \sum_{j=0}^{\al +2}\frac{(\al+3-j)_{2j}}{j!(j+1)!} I_{4,j}^\al(f,g) -N(1+M)f'(0)g'(0).
    \label{eq3.19}
    \end{aligned} 
    \end{equation}
The right-hand side of (\ref{eq3.19}) is symmetric w. r. t. the functions $f,g$ , and so we can interchange their roles in the scalar product on the left. This completes the proof of part (i) in Theorem \ref{thm3.1}.

\textbf{(ii)} \;The orthogonality relation (\ref{eq3.2}) is a simple consequence of part (i), since for all $n \ne m$, $n,m \in \mathbb{N}_0$, the difference of the eigenvalues $\la_m^{\al,M,N}-\la_n^{\al,M,N}$ does not vanish, while
   \begin{equation*}
  \begin{aligned} 
 \hspace{1cm}\big(\la_m^{\al,M,N}&-\la_n^{\al,M,N} \big)\big(y_n,y_m\big)_{w(\al,M,N)}\\
 &=\big(\mathcal{L}_x^{\al,M,N}y_n,y_m\big)_{w(\al,M,N)}-  \big(y_n,\mathcal{L}_x^{\al,M,N}y_m\big)_{w(\al,M,N)}=0.\hspace{3cm} \square
    \end{aligned} 
    \end{equation*}
 \begin{corollary}
 \label{cor3.4}
 For any real-valued function $f \in C^{(4\al +10)}[0,\infty)$, there holds 
 \begin{equation}
   \begin{aligned} 
 \big(-\mathcal{L}_x^{\al,M,N}f&,f\big)_{w(\al ,M,N)}\ge
  \frac{1}{\al!} \int_{0}^{\infty}e^{-x}x^{\al +1}[f'(x)]^2dx+\\
  &+\frac{M}{\al!(\al +2)!} \int_{0}^{\infty}e^{-x}\big\lbrace D_x^{\al +2} \big\lbrack x^{\al +1}f(x)\big\rbrack \big\rbrace^2 dx +N(1+M)[f'(0)]^2,
     \label{eq3.20}
     \end{aligned} 
     \end{equation}
 and equality is attained for any linear function $f$.
 \end{corollary}
 \begin{proof}
Inequality (\ref{eq3.20}) follows from identity (\ref{eq3.19}) since $I_{3,j}^\al(f,f) \ge 0$, $1 \le j \le 3$, and $I_{4,j}^\al(f,f) \ge 0$, $0 \le j \le \al+2$. All these integrals clealy vanish, if $f$ is a linear function.
 \end{proof}

\vskip0.5cm
\begin{footnotesize}
\noindent
C. Markett, Lehrstuhl A f\"ur Mathematik, RWTH Aachen, 52056 Aachen, Germany;
\sPP
E-mail: {\tt markett@matha.rwth-aachen.de}

\end{footnotesize}


\begin{thebibliography}{10}

\bibitem{Ba1} H. Bavinck, 
	{\em  Differential operators having Sobolev-Laguerre polynomials as eigenfunctions: new developments}, J. Comput. Appl. Math. 133 (2001), 183--193.
	 
\bibitem{Ba2} H. Bavinck, 
	{\em  Differential operators having Sobolev-type Gegenbauer polynomials as eigenfunctions}, J. Comput. Appl. Math. 118 (2000), 23--42.

\bibitem{BK} H. Bavinck, J. Koekoek,
	{\em  Differential operators having symmetric orthogonal polynomials as eigenfunctions}, J. Comput. Appl. Math. 106 (1999), 369--393.
	
\bibitem{DI} A. J. Dur\'{a}n, M. D. de la Iglesia, 
	{\em  Differential equations for discrete Laguerre-Sobolev orthogonal polynomials}, J. Approx. Theory 195 (2015), 70--88. 
	
\bibitem{HTF2} A. Erd\'{e}lyi, W. Magnus, F. Oberhettinger, F.G. Tricomi,
  	{\em Higher Transcendental Functions, vol. II}, McGraw-Hill, 1953.
  	
\bibitem{EKLW} W. N. Everitt, K. H. Kwon, L. L. Littlejohn, R. Wellman,
	{\em Orthogonal polynomial solutions of linear ordinary differential equations}, J. Comput. Appl. Math. 133 (2001), 85--109.
	
\bibitem{GHH} A. Gr\"{u}nbaum, L. Haine, E. Horozov,
	{\em Some functions that generalize the Krall-Laguerre polynomials}, J. Comput. Appl. Math. 106 (1999), 271--297.
	
\bibitem{IL} P. Iliev, 
	{\em Krall-Laguerre commutative algebras of ordinary differential operators}, Ann. Mat. Pura Appl. 192 (2013), 203--224.

\bibitem{JKLL} I. H. Jung, K. H. Kwon, D. W. Lee, L. L. Littlejohn,
	{\em Sobolev orthogonal polynomials and spectral differential equations}, Trans. Amer. Math. Soc. 347 (1995), 3629--3643.

\bibitem{KK1} J. Koekoek, R. Koekoek, 
	{\em On the differential equation for Koornwinder's generalized Laguerre polynomials}, Proc. Amer. Math. Soc. 112 (1991), 1045--1054.

\bibitem{KKB} J. Koekoek, R. Koekoek, H. Bavinck, 
	{\em On differential equations for Sobolev-Laguerre polynomials}, Trans. Amer. Math. Soc. 350 (1998), 347--393.
		
\bibitem{K1} R. Koekoek, 
	{\em The search for differential equations for orthogonal polynomials by using computers}, Delft University of Technology, Report 91-55, 1991.

\bibitem{K2} R. Koekoek, 
	{\em The search for differential equations for certain sets of orthogonal polynomials}, J. Comput. Appl. Math. 49 (1993), 111--119.

\bibitem{KM1} R. Koekoek, H. G. Meyer,
	{\em  A generalization of Laguerre polynomials}, Delft University of Technology, Report 88-28, 1988.
	
\bibitem{KM2} R. Koekoek, H. G. Meyer,
	{\em  A generalization of Laguerre polynomials}, SIAM J. Math. Anal. 24 (1993), 768--782. 
	
\bibitem{Ko} T. H. Koornwinder, 
	{\em Orthogonal polynomials with weight function $(1-x)^{\alpha}(1+x)^{\beta}+M\delta(x+1)+N\delta(x-1)$}, Canad. Math. Bull. 27(2) (1984), 205--214. 

\bibitem{MX} F. Marcell\'{a}n, Y. Xu, 
	{\em On Sobolev orthogonal polynomials}, Expo. Math. 33 (2015), 308--352.
 	
\bibitem{Ma1} C. Markett, 
	{\em An elementary representation of the higher-order Jacobi-type differential equation}, Indag. Math. 28 (2017), 976--991. 
 	
\bibitem{Ma2} C. Markett, 
	{\em On the higher-order differential equations for the generalized Laguerre polynomials and Bessel functions}, submitted. See also arXiv:1708.00230v1 [math.CA], Aug 2017.
	
\bibitem{Ma3} C. Markett, 
	{\em The higher-order differential operator for the generalized Jacobi polynomials -- new representation and symmetry}, Indagationes Mathematicae (2018),\\ https://doi.org/10.1016/j.indag.2018.08.004.
	
\bibitem{Wh} F. J. W. Whipple, 
	{\em A group of generalized hypergeometric series: Relations between 120 allied series}, Proc. London Math. Soc. 23 (1925), 104--114. 
	
\end{thebibliography}
\end{document}